\documentclass[12pt,a4paper,reqno,oneside]{amsart}

\usepackage{csquotes}

\usepackage{vmargin}
\setpapersize{A4}
\setmargins
{2.5cm}		% margen izquierdo
{1.5cm}		% margen superior
{16cm}		% ancho del texto
{23.42cm}	% altura del texto
{10pt}		% altura de los encabezados
{1cm}		% espacio entre el texto y los encabezados
{0pt}		% altura del pie de pagina
{2cm}		% espacio entre el texto y el pie de pagina

\usepackage[dvipsnames]{xcolor} % Must come before tikz!
\colorlet{cite}{LimeGreen!50!Green}

\usepackage[%
bookmarks=true,			% show bookmarks bar?
unicode=true,			% non-Latin characters in Acrobat’s bookmarks
colorlinks=true,		% false: boxed links; true: colored links
linkcolor=Blue,			% color of internal links
citecolor=cite,			% color of links to bibliography
filecolor=magenta,		% color of file links
urlcolor=RoyalBlue		% color of external links
]{hyperref}				% One of the most useful packages I have found.

\newtheoremstyle{mydef}
  {}		% Space above environment
  {}		% Space below environment
  {}		% Body font
  {}		% Indent amount (empty = no indent, \parindent = para indent)
  {\scshape}	% theorem head font
  {. }		% Punctuation after heading
  { }		% Space after heading
  {\thmname{#1}\thmnumber{ #2}\thmnote{ #3}}	% Heading spec

\theoremstyle{plain}	
\newtheorem{theorem}{Theorem}[section] 
\newtheorem{lemma}[theorem]{Lemma} 

\newtheorem*{theorem*}{Theorem}

\theoremstyle{mydef} 
\newtheorem*{conjecture*}{Conjecture}

\theoremstyle{remark}

\newtheorem*{proposition*}{Proposition}
\newtheorem*{lemma*}{Lemma}
\newtheorem*{corollary*}{Corollary}
\theoremstyle{definition}
\newtheorem{example}[theorem]{Example}
\newtheorem{definition}[theorem]{Definition}
\newtheorem{remark}[theorem]{Remark}

\theoremstyle{remark}
\newtheorem{claim}[theorem]{Claim}

\usepackage{tikz, cancel}  % for diagrams
\usetikzlibrary{matrix,arrows,arrows.meta,
decorations.markings,decorations.pathmorphing,calc,intersections,shapes,backgrounds,positioning}
\usepackage{tikz-cd}

\DeclareMathOperator{\Diag}{Diag}

\DeclareMathOperator{\SL}{\mathrm{SL}}
\DeclareMathOperator{\SU}{SU}
\DeclareMathOperator{\UU}{U}

\DeclareMathOperator{\Ad}{Ad}

\DeclareMathOperator{\Mon}{Mon}
\DeclareMathOperator{\Div}{Div}
\DeclareMathOperator{\LG}{LG}
\DeclareMathOperator{\Fuk}{Fuk}

\renewcommand{\S}{\mathrm{S}}

\newcommand{\ce}{\mathrel{\mathop:}=}

\usepackage{fdsymbol}

\title[Birationally equivalent Landau--Ginzburg models]
{Birationally equivalent Landau--Ginzburg models\\ 
on cotangent bundles and adjoint orbits}
\author{ Bruno Suzuki}
\begin{document}

\begin{abstract}{
We show that the Lie potential on the minimal semisimple adjoint orbit 
$\mathcal O_n$ of $\mathfrak{sl}(n+1,\mathbb C)$ coincides with toric potential on $T^*\mathbb P^{n}$.
We then study the corresponding Landau--Ginzburg models in deformation families and give some  examples of how the deformations 
affect the mirrors. }
\end{abstract}

\maketitle

\tableofcontents

\section{Strategy}

We show that the Lie theoretical potential   on the minimal semisimple adjoint orbit 
$\mathcal O_n$ of $\mathfrak{sl}(n+1,\mathbb C)$  as defined in \cite{GGS1} 
coincides with standard  potential on $T^*\mathbb P^{n}$ associated  with the Hamiltonian
action of a 1-dimensional algebraic torus $\mathbb C^*$. 
More precisely, we show that $(T^*\mathbb P^{n},w)$ and $(\mathcal O_n, w)$
are birationally equivalent.
Having described this coincidence, 
we are then able to write a potential on the  deformation family containing both $T^*\mathbb P^n$ and $\mathcal O_n$, 
so that it can be upgraded to a deformation of Landau--Ginzburg models.
We then explore specific examples of  Landau--Ginzburg models  and give examples of how their deformations 
affect their mirrors.
The proposed strategy here is to consider both deformations of varieties and deformations of the potential,
combining them to describe  deformations of LG models, obtaining families
$$\LG \longrightsquigarrow \LG'$$ 
where $\LG$ and $\LG'$ behave very differently dualitywise.
Indeed, we will have examples $\LG_0= (T^*\mathbb P^1, w)$ that is a selfdual LG model defined on a 
toric variety which is not affine, deforming to $\LG_2= (\mathcal O_1, w)$ which is a nonselfdual LG model defined 
on a nontoric affine variety.

 Landau--Ginzburg models on  adjoint orbits of semisimple Lie algebras $f_H\colon \mathcal{O}(H_0) \rightarrow \mathbb C$ 
 were constructed by 
 Gasparim--Grama--San Martin in \cite{GGS1,GGS2}.
	We consider here the minimal case (as in Definition \ref{minimal}) 
which we denote by $\LG_n$.
These Landau--Ginzburg models  were shown to satisfy the conjecture of 
Katzarkov, Kontsevich and Pantev about three new Hodge theoretical invariants that take into consideration the potential, 
see \cite{BGRS2, GR} and also \cite{BCG} for some more features about Hodge diamonds of compactifications.

It would be desirable to carry out the study of all deformations families involving 
 semisimple adjoint orbits of $\mathfrak{sl}(n+1,\mathbb C)$. However, here we will carry out details only for the 
 case of minimal orbits, presenting more details for the case of $\mathfrak {sl}(2,\mathbb C)$,
which is already nontrivial, as we shall see. Indeed,
 \cite{BBGGS} showed that $\LG_2$ has no projective mirrors, 
 and furthermore, proved  that the Fukaya--Seidel category of $\LG_2$ is equivalent to a proper subcategory of the category of coherent sheaves on the second Hirzebruch surface, 
 which however is not realised by any complex subvariety. 
As a consequence of 
 \cite{BBGGS} we know that there does not exist any projective variety whose category of coherent sheaves
 is equivalent to the Fukaya--Seidel category of $\LG_2$. 
 However, in \cite{G}, using the intrinsic mirror symmetry recipe of Gross and Siebert \cite{GS}, Elizabeth Gasparim
 constructed 
 in \cite{G} a
 Landau--Ginzburg mirror for $\LG_2$ which is geometrically extremely different from $\LG_2$. 
 This nontrivial behaviour of duality of $\LG_2$ will be contrasted to the selfduality of $\LG_0$.

\section{Deformations}

\begin{definition}
A \textbf{Landau--Ginzburg  model} (LG)  is  a pair $(X,w)$ formed by a  variety
$X$ and a holomorphic  function $w\colon X \rightarrow \mathbb C$ (or $\mathbb P^1$) called the superpotential.
\end{definition}

To expand our families of examples, we will consider pairs $(X,R)$ where $R\colon X \rightarrow \mathbb C$ 
is a rational function, defined as $R= f/g$,  then this rational function can be considered as a map to  
$\mathbb P^1 $ simply by making $R = [f:g]$, see \cite{BGGS}.
Then, such a map can then be promoted to an LG model by blowing up the indeterminacy locus $\mathcal I$ of $R$, that is 
the set of points where $f=g=0$, and then
$R$ can be extended to  a holomorphic map on $\widetilde{X}$,  the blow-up of $X$ at $\mathcal{I}$.
 Therefore in the examples that follow, we wish to consider {rational potentials}.
 
\begin{definition} 
A \textbf{rational Landau--Ginzburg  model}   is  a pair $(X,R)$ formed by a  variety
$X$ and a rational function $R\colon X \rightarrow \mathbb  P^1$ called the superpotential.
\end{definition}
 
From now on, by a LG model we mean a rational Landau--Ginzburg model (note that this notion 
includes all LG models that have a holomorphic potential). 

\begin{definition}
A (commutative) \textbf{deformation family} of a Landau--Ginzburg  model  $(X,w)$ is a smooth family of 
 Landau--Ginzburg models $(X_t, w_t)$, with $t \in D \subset \mathbb C^n$ an open ball containing $0$,  
such that $(X_0, w_0) = (X,w)$.
We call $X_t$ a \textbf{deformation} of $X_0$, 
denoted by 
$(X_0, w_0) \longrightsquigarrow (X_t,w_t)$.
\end{definition}

For Landau--Ginzburg models considered in this text, the variety $X$ will be a noncompact Calabi--Yau. 
This is exactly the case for which Katzarkov, Kontsevich and Pantev proved smoothness 
of the deformation spaces \cite{KKP}.

\begin{example} \label{ex1}
Consider $T^*\mathbb{P}^1$ with coordinates $([1,x],y)$ and the Landau--Ginzburg models $\LG_0 = (T^*\mathbb{P}^1, w_0)$, 
$\LG_1 = (T^*\mathbb{P}^1, w_1)$
given by the functions 
$w_0 = x+y+\frac{y^2}{x}$ and
$w_1 = 2x$.

Using the potential 
$w_t= 2x+t(x+y+\frac{y^2}{x})$  
on $X_t = T^*\mathbb{P}^1$
we obtain the deformation 
$$\LG_0 \longrightsquigarrow \LG_1.$$

\end{example}

\begin{definition} 
Let $G$ be a Lie group and $\mathfrak g$ its Lie algebra. The {\bf adjoint orbit} of 
an element $H_0 \in \mathfrak g$ is 
$$\mathcal{O}\left( H_0\right) =\mathrm{Ad}(G)\cdot H_0=
\{g H_0 g^{-1}:g\in G\}.$$
\end{definition}

In all examples considered here $H_0$ is a diagonal matrix. 
Since we work with complex semisimple Lie groups $G$, the orbit $\Ad(K)\cdot H_0$ of a compact subgroup $K < G$ is a \textbf{flag manifold } isomorphic to a homogeneous manifold $G/P$, where $P$ is a parabolic subgroup of $G$.

Following 
\cite[Def.\thinspace 1.1]{BGRS2}, 
we have:
\begin{definition}\label{minimal}
Let $\mathcal O_n$ the adjoint orbit of $H_0=\Diag(n,-1,\ldots,-1)$ in $\mathfrak{sl}(n+1,\mathbb{C})$, we call it the {\bf minimal orbit}. \end{definition}

The minimal adjoint orbit $\mathcal O_n$ is diffeomorphic to the cotangent bundle of the projective space 
 $\mathbb{P}^{n}$ \cite[Thm.\thinspace2.1]{GGS2}. $\mathcal O_n$ is a nontoric Calabi--Yau manifold.

\begin{example} \label{ex2}
Let $\LG_2 = (\mathcal{O}_1, f_H)$, 
where $\mathcal{O}_1$ is the semisimple adjoint orbit of the Lie algebra $\mathfrak{sl}(2,\mathbb{C})$,
obtained by choosing $H_0=\Diag(1,-1)$. 

Lemma \ref{deform.fam} describes 
$\mathcal{O}_1$  as a deformation  of $T^*\mathbb{P}^1$,
explicitly presenting the  deformation family.
We put on $\mathcal{O}_1$ the potential given by
$f
\begin{pmatrix}
x&y\\
z&-x
\end{pmatrix}
 = 2x.$
Observe that this potential coincides  with the potential $w$ defined on $T^*\mathbb P^1$ in Example \ref{ex1},
forming $LG_1$.
Defining the same potential $2x$ on every element of the family, we obtain a deformation of Landau--Ginzburg models
$$\LG_1 \longrightsquigarrow \LG_2.$$
We will present a more detailed approach to this deformation family in Example \ref{partial},
including a partial compactification.
\end{example}

\begin{remark}
 The choice of potential on $\mathcal O_1$ comes from the following construction.
 We choose 
$H=H_0= \left(\begin{matrix} 1 & 0 \\ 0 & -1 \end{matrix}\right),$
and then the adjoint orbit 
$$\mathcal O_1= \mathrm{Ad}(\mathrm{SL}(2)) H_0= \left\{gH_0g^{-1}: g \in \mathrm{SL}(2)\right\}$$
 is the (smooth) affine hypersurface in $\mathbb{C}^3$ given by the equation
$x^2+yz-1=0.$	
The potential $f_H\colon \mathcal O_1 \rightarrow \mathbb C$, defined by
$ A \mapsto \langle H,A\rangle$ gives simply $f_H(A) = 2x$.

	As a real manifold $\mathcal{O}_1$ is diffeomorphic to  $T^*\mathbb{P}^1$.
	Now we recall a lemma from \cite{BBGGS} which shows the existence of Lagrangian submanifold of the adjoint orbit of $\mathfrak{sl}(2)$ 
that is diffeomorphic to $\mathbb P^1$ 
(observe this is not granted a priori since the symplectic structure of $\mathcal{O}(H_0)$ is not that of $T^* \mathbb P^1$)
was proved in 
\cite[Lem.\thinspace2.1]{BBGGS}.	However, $\mathcal O_1$ is affine, hence it does not contain an algebraic subvariety 
isomorphic to $\mathbb P^1$.
\end{remark}

\begin{example} \label{ex3}
Observe that in the Example \ref{ex1} we varied the potential while keeping the underlying manifold fixed, 
whereas in Example \ref{ex2} we varied the manifold while keeping the potential.
Combining them, we obtain a deformation that changes both, namely
$$
\LG_0 \longrightsquigarrow \LG_2.
$$
\end{example}

We then wish to compare the mirrors of $\LG_0$ and $\LG_2$, and we will see 
that they behave very differently. Since $T^*\mathbb P^1$ is a toric variety, 
we can use toric duality for $\LG_0$.

\section{Toric duality}

\begin{definition}
An $n$-dimensional variety $X$ is called \textbf{toric} if it satisfies the following two properties: 

\begin{itemize}
\item there exists an action of the torus  $(\mathbb{C}^*)^n \times X \rightarrow X$, and
\item the Zariski closure of $X$ is the torus:  $\overline{(\mathbb{C}^*)^n}= X.$
\end{itemize} 
\end{definition}

For basic notions of toric varieties see \cite{cox}.
The polytope representing a toric variety is the convex figure obtained by drawing the complement of the torus, i.e., 
$\overline{X} - (\mathbb{C}^*)^n$. More precisely, it is the image of $\overline{X}$ by the moment map, 
which is, by the theorems of Atiyah and Guillemin--Sternberg the convex hull of the image of the fixed points of the torus action.

\begin{example}
$\mathbb{P}^1 \times \mathbb{P}^1$ 
is depicted by a square
\begin{tikzpicture}[local bounding box=bb,baseline=(bb.center)]

\filldraw[Thistle] (0,0) -- ++(1,0) -- ++(0,1) -- ++(-1,0) -- ++(0,-1);
\foreach \x in {0,...,1}
		{	\foreach \y in {0,...,1}
			{
			\filldraw (\x,\y) circle (0.03cm);
			}
		}
\draw (0,0) -- ++(1,0) -- ++(0,1) -- ++(-1,0) -- ++(0,-1);
\end{tikzpicture}

\end{example}

\begin{example}
$\mathbb{P}^2$ 
is depicted by a triangle
\begin{tikzpicture}[local bounding box=bb,baseline=(bb.center)]

\filldraw[Thistle] (0,0) -- ++(1,0) -- ++(-1,1) -- ++(0,-1);
\foreach \x in {0,...,1}
		{	\foreach \y in {0,...,1}
			{
			\filldraw (\x,\y) circle (0.03cm);
			}
		}
\draw (0,0) -- ++(1,0) -- ++(-1,1) -- ++(0,-1); 
\end{tikzpicture}
\end{example}

\begin{example}
$T^*\mathbb{P}^1$
is depicted by 
\begin{tikzpicture}[local bounding box=bb,baseline=(bb.center)]
\filldraw[Thistle] (0,1) -- ++(0,-1) -- ++(1,0) -- ++(2,1);

\foreach \x in {0,...,3}
		{	\foreach \y in {0,...,1}
			{
			\filldraw (\x,\y) circle (0.03cm);
			}
		}

\draw (0,1) -- ++(0,-1) -- ++(1,0) -- ++(2,1);

\foreach \i in {0,...,3}
	{
	\filldraw[white] (\i,1) circle (0.02cm);
	}	
\end{tikzpicture}
\end{example}

	For the case when $X$ is a toric variety, 
Clarke \cite{clarke} showed that one can state a generalised version of the Homological Mirror Symmetry conjecture of Kontsevich \cite{Kon}
as a duality between LG models. 
He also showed that this construction generalises the dualities due to Batyrev--Borisov \cite{BB}, 
Berglung--H\"ubsch \cite{BH}, Givental \cite{G}, and Hori--Vafa \cite{HV}. 
	Given toric LG model $(X,f)$, the construction in \cite{clarke} produces a dual LG model $(X^\vee, f^\vee)$. 
	
We recall some definitions from \cite{clarke}.
Associated to a toric
Landau--Ginzburg model is its linear data:
\begin{enumerate}
\item the linear data $\Div$, and
\item the linear data $\Mon$.
\end{enumerate}
 \cite{clarke} uses a more refined version of this data,
which includes the compactified K\"ahler class and for duality requires the kopasectic condition, but we will not need such details here. 
The dual toric Landau-Ginzburg model is obtained by exchanging $\Div$ and $\Mon$, that is 
\[
\Div(X) = \Mon(f^\vee), \quad \quad \Div(X^\vee) = \Mon(f).
\] 
	
The matrix $\Mon$ describes the infinitesimal action on monomials, for details see \cite{clarke}.
The matrix $\Div$ encodes the divisor of the character-to-divisor map. In particular, the Chow group of $X$ is the 
cokernel of $\Div(X)$. In our examples the rows of $\Div$ will represent the divisors of $X$ and the rows of $\Mon$ will represent the monomials appearing in the potential $f$. 

\begin{definition}
If there is an isomorphism $(X^\vee, f^\vee)\simeq (X,f)$, then the LG model is called \textbf{selfdual}. 
\end{definition}

We now give some elementary examples of toric duality.
	Using the work of Clarke in \cite{clarke}, Callander, Gasparim, Jenkins and Silva \cite{CGJS} considered
mirror symmetry for Landau--Ginzburg models defined over vector bundles on projective spaces, 
showing that both the cotangent bundle of $\mathbb P^1$ and the resolved conifold admit potentials that form a self-mirror LG model.

\begin{example}[$\LG_0 \Leftrightarrow \LG_0$] \label{selfdual}
For the case of $T^*\mathbb P^1$ using the potential $x+y+\frac{y^2}{x}$ one obtains a LG model
	that is dual to itself  \cite[Prop.\thinspace6.2]{CGJS} as depicted in Figure \ref{self}. Self-duality of this LG model 
	is verified by simply pointing out that in this case the toric data is
	$$\Mon = \Div= \left(\!\!\!
	\begin{array}{rr} 
	1& 0 \\
	-1 &2\\
	0&1
	\end{array} \!\right).$$

\begin{figure}[h]
\centering
\begin{tikzpicture}

\filldraw[Thistle] (0,1) -- ++(0,-1) -- ++(1,0) -- ++(2,1);

\foreach \x in {0,...,3}
		{	\foreach \y in {0,...,1}
			{
			\filldraw (\x,\y) circle (0.03cm);
			}
		}

\draw (0,1) -- ++(0,-1) -- ++(1,0) -- ++(2,1);

\foreach \i in {0,...,3}
	{
	\filldraw[white] (\i,1) circle (0.02cm);
	}

\node at (0.5,-3) {$x + y + \dfrac{y^2}{x}$};
\node at (8.5,-3) {$x + y + \dfrac{y^2}{x}$};

\draw[->] (0.5,-0.5) to [out=-60,in=120] (8.5,-2.7);

\draw[<-, dashed] (8.5,-0.5) to [out=240,in=60] (0.5,-2.7);

\begin{scope}[shift={(6,0)}] 
\filldraw[Thistle] (0,1) -- ++(0,-1) -- ++(1,0) -- ++(2,1);

\foreach \w in {0,...,3}
		{	\foreach \z in {0,...,1}
			{
			\filldraw (\w,\z) circle (0.03cm);
			}
		}

\draw (0,1) -- ++(0,-1) -- ++(1,0) -- ++(2,1);

\foreach \i in {0,...,3}
	{
	\filldraw[white] (\i,1) circle (0.02cm);
	}	
\end{scope}

\end{tikzpicture}
\caption{Selfdual $\LG_0$ model}\label{self}
\end{figure}
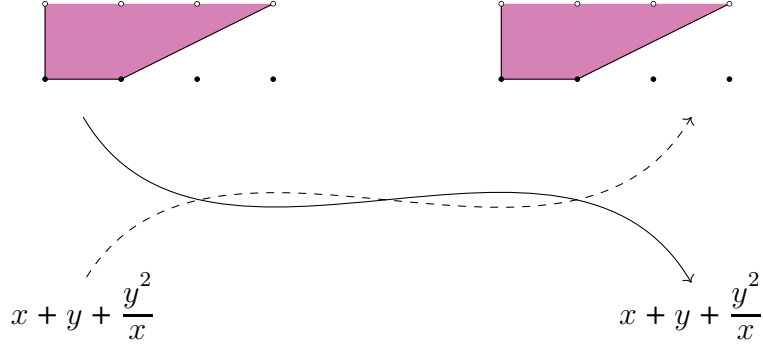
\end{example}

\begin{example}(A nontrivial duality) We now see a nontrivial duality $(Y,g) = (X,f)^\vee$.
Consider the Landau--Ginzburg models  
$(X,f) =(\mathbb P^2,  x + y + x^{-1}+y^{-1})$ and $(Y,g)=(\mathbb P^1 \times \mathbb P^1,x + y + x^{-1}y^{-1})$
 as depicted in Figure \ref{nondual}.

Since the $\Div$ matrix is given by the inward normals of the moment polytope, we have:
\[
\Div_{X} = 
\left(\!\!
\begin{array}{rr}
1 & 0 \\
0 & 1 \\
-1 & -1
\end{array}
\right), \quad
\Div_{Y} =
\left(\!\!
\begin{array}{rr}
1 & 0 \\
0 & 1 \\
-1 & 0 \\
0 & -1
\end{array}
\right).
\]
To see that  $(Y,g)$ is dual to $(X,f)$ just observe that $\Mon_{g}=\Div_{X}$, %
and  $\Div_{Y} = \Mon_{f}$.
\end{example}

\begin{figure}[h]
\centering
\begin{tikzpicture}

\filldraw[Thistle] (0,0) -- ++(1,0) -- ++(0,1) -- ++(-1,0) -- ++(0,-1);
\filldraw[Thistle] (6,0) -- ++(1,0) -- ++(-1,1) -- ++(0,-1);

\foreach \x in {0,...,1}
		{	\foreach \y in {0,...,1}
			{
			\filldraw (\x,\y) circle (0.03cm);
			}
		}
\draw (0,0) -- ++(1,0) -- ++(0,1) -- ++(-1,0) -- ++(0,-1);

\foreach \w in {6,...,7}
		{	\foreach \z in {0,...,1}
			{
			\filldraw (\w,\z) circle (0.03cm);
			}
		}

\draw (6,0) -- ++(1,0) -- ++(-1,1) -- ++(0,-1);

\node at (0.5,-4) {$x + y + \dfrac{1}{xy}$};
\node at (6.5,-4) {$x + y + \dfrac{1}{x} + \dfrac{1}{y}$};

\draw[->] (0.5,-0.5) to [out=-40,in=140] (6.5,-3.5);

\draw[->, dashed] (0.5,-3.5) to [out=40,in=220] (6.5,-0.5);

\end{tikzpicture}
\caption{Nontrivial toric duality}\label{nondual}
\end{figure}
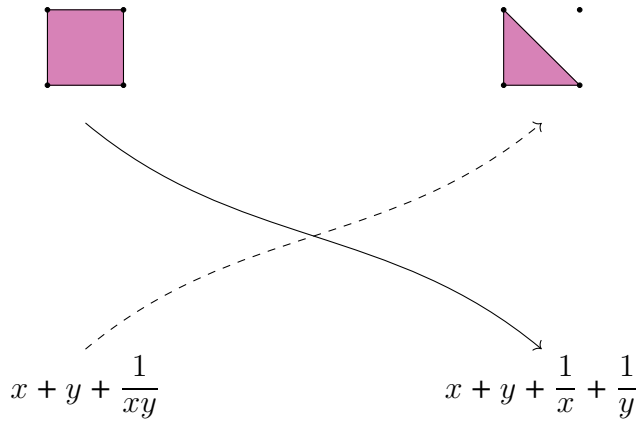

\begin{example}(Hyperk\"ahler families)
The birational isomorphism considered here could be considered in 
a much broader context. In fact, if instead of considering cotangent bundles of projective spaces, 
we considered cotangent bundles of generalised flag manifolds, then we could consider
families of hyperk\"ahler structures; we recall the definition.

A hyperk\"ahler structure on the Riemannian manifold $(M,g)$ is expressed using three almost complex structures $I, J, K$  satisfying 
\[
I^2 = J^2 = K^2 = IJK = -1,
\]
such that any linear combination $aI+bJ+cK$ with $a^2+b^2+c^2=1$ is a K\"ahler structure on $M$, whose corresponding complex structures may be regarded as a deformation of $M$.
All hyperk\"ahler manifolds are Ricci-flat and consequently Calabi--Yau manifolds. 

Families of hyperk\"ahler structures containing both cotangent bundles of flag manifolds and
 adjoint orbits   were studied by Kronheimer \cite{Kro1}, Biquard \cite{Biq}, and Kovalev \cite{Ko},
 we briefly recall some of their results.
Let $G$ be a complex semisimple Lie group with Lie algebra $\mathfrak g$ and let $H$ denote the stabiliser 
of an element of $\mathfrak g$.
Kovalev's theorem \cite[Thm.\thinspace1.1]{Ko}\label{teo1.1Kovalev} 
describes a family of hyperk\"ahlerian structures on $G/H$.
In particular, this family contains a complex structure for which the quotient $G/H$ is isomorphic to a semisimple adjoint orbit of $G$.
More explicitly,
the adjoint orbit $\Ad(G) \cdot H_0$ occurs as a deformation of the cotangent bundle of the flag manifold $\Ad(K) \cdot H_0$, where $K$ is a compact subgroup of $G$.
	
Many interesting examples can be obtained using the adjoint orbit $\mathcal{O}(H_0)$ of an element 
of the Cartan subalgebra $H_0 \in \mathfrak h \subset \mathfrak g$. 
Suppose we take $\mathfrak g =\mathfrak{sl}(n,\mathbb C)$, hence
we would have a flag manifold 
$$\mathbb F= \mathrm{Ad}(\SU(n))\cdot H_0$$ of $\SU(n)$, 
and a diffeomorphism
$$\mathcal O(H_0) = \mathrm{Ad}(\SL(n))\cdot H_0\simeq  T^*\mathbb F$$ 
to the cotangent bundle of $\mathbb F$, where the flag manifold is isomorphic a homogeneous manifold  of the form
$$\mathbb F={\SU(n)}/{\S(\UU(d_1)\times \UU(d_2) \times \cdots \times \UU(d_k))}$$
 with $d_1+d_2+\cdots +d_k=n,$ see \cite{Arv}.
 \end{example}

\section{LG models from Lie theory}
In this section we develop the description of the Lie theoretical potential on adjoint orbits in local coordinates,
so that it can be compared with the toric potential.
We first recall some results from  \cite{GGS1} and  \cite{GGS2}.
	Let $\mathfrak{g}$ be a complex semisimple Lie algebra with Cartan subalgebra $\mathfrak{h}$, and $\mathfrak{h}_{\mathbb{R}}$ the real subspace generated by the roots of $\mathfrak{h}$.
	Let $\Pi$ be the set of all roots of $\mathfrak g$.
	An element $H\in\mathfrak{h}$ is  called {regular} if $\alpha \left(
H\right) \neq 0$ for all $\alpha \in \Pi $.

 Choose $H_{0}\in \mathfrak{h}$ and a regular element $H\in \mathfrak{h}_{\mathbb{R}}$,
 the {\bf height function} 
 $f_{H}\colon \mathcal{O} \left( H_{0}\right) \rightarrow \mathbb{C}$
 with respect to $H$, 
 defined by
  \[
    f_{H}\left( x\right) = \langle H,x\rangle, \qquad x\in \mathcal{O}\left(H_{0}\right),
  \]
 has a finite number of isolated singularities and defines a symplectic Lefschetz fibration on $\mathcal{O}(H_0)$, 
\cite[Thm.\thinspace 3.1]{GGS1}. It is this heigh function that will give rise to the expression of what we 
call the Lie potential. \\

  We now describe the potential on the adjoint orbit in detail, 
writing its description in each coordinate chart by a polynomial map. 
  The construction of this intermediate step also clarifies why the potential gives rise to a Lefschetz fibration,
 which is by definition locally a quadratic polynomial map.

	Given $H_{0}\in \mathfrak{h}_{\mathbb{R}}$, let
  	\begin{equation*}
  	\mathfrak{n}_{H_{0}}^{+}=\sum_{\alpha \left( H_{0}\right) >0}\mathfrak{g}%
  	_{\alpha }\qquad \mathrm{and}\qquad \mathfrak{n}_{H_{0}}^{-}=\sum_{\alpha
  		\left( H_{0}\right) <0}\mathfrak{g}_{\alpha }
  	\end{equation*}%
  	be the sums of the eigenspaces of $\mathrm{ad}\left(H_{0}\right) $
  	associated to the positive and negative eigenvalues, respectively. 
  	The subspaces $\mathfrak{n}_{H_{0}}^{\pm }$ are nilpotent subalgebras.
  	If $G$ is a Lie group with  Lie algebra $\mathfrak{g}$ then the subgroups
  	$N_{H_{0}}^{\pm }=\exp \mathfrak{n}_{H_{0}}^{\pm }$ 
  	are closed and the set
  	\begin{equation*}
  	N_{H_{0}}^{-}Z_{H_{0}}N_{H_{0}}^{+}=N_{H_{0}}^{-}N_{H_{0}}^{+}Z_{H_{0}}
  	\end{equation*}%
  	is open and dense in $G$ where $Z_{H_{0}}$ is the centraliser of  $H_{0}$
  	in $G$ and the product
  	\begin{equation*}
  	\left( y,x,h\right) \in N_{H_{0}}^{-}\times N_{H_{0}}^{+}\times
  	Z_{H_{0}}\mapsto yxh\in N_{H_{0}}^{-}N_{H_{0}}^{+}Z_{H_{0}}
  	\end{equation*}%
  	is a diffeomorphism.
  	Consider now the adjoint orbit $\mathcal{O}\left( H_{0}\right) =%
  	\mathrm{Ad}\left( G\right) H_{0}=G/Z_{H_{0}}$. The set  
  	$%
  	\mathrm{Ad}\left( N_{H_{0}}^{-}N_{H_{0}}^{+}Z_{H_{0}}\right) H_{0}$,
  	denoted just by  $N_{H_{0}}^{-}N_{H_{0}}^{+}Z_{H_{0}}\cdot H_{0}$
  	satisfies
  	\begin{equation*}
  	N_{H_{0}}^{-}N_{H_{0}}^{+}Z_{H_{0}}\cdot
  	H_{0}=N_{H_{0}}^{-}N_{H_{0}}^{+}\cdot H_{0}
  	\end{equation*}%
  	since $Z_{H_{0}}\cdot H_{0}=H_{0}$. Given that  $N_{H_{0}}^{-}N_{H_{0}}^{+}Z_{H_{0}}$
  	is open and dense in  $G$ it follows that  $N_{H_{0}}^{-}N_{H_{0}}^{+}\cdot H_{0}$
  	is open and dense in  $\mathcal{O}\left( H_{0}\right) $. 
  	Moreover, the map 
  	\begin{equation*}
  	\left( y,x\right) \in N_{H_{0}}^{-}\times N_{H_{0}}^{+}\mapsto yx\cdot
  	H_{0}\in N_{H_{0}}^{-}N_{H_{0}}^{+}\cdot H_{0}
  	\end{equation*}%
  	is a diffeomorphism.
  	
  	On the other hand, $\exp\colon \mathfrak{n}_{H_{0}}^{\pm }\rightarrow N_{H_{0}}^{\pm
  	}$ is a diffeomorphism. Therefore,
  	\begin{equation*}
  	\left( Y,X\right) \in \mathfrak{n}_{H_{0}}^{-}\times \mathfrak{n}%
  	_{H_{0}}^{+}\mapsto e^{\mathrm{ad}\left( Y\right) }e^{\mathrm{ad}\left(
  		X\right) }\cdot H_{0}\in N_{H_{0}}^{-}N_{H_{0}}^{+}\cdot H_{0}
  	\end{equation*}%
  	defines a coordinate chart of  $\mathcal{O}\left( H_{0}\right) $, around
  	$H_{0}$ (with open and dense image).
  	
  	The goal here is to describe the  Lefschetz fibration on  $\mathcal{O}%
  	\left( H_{0}\right) $ defined by  $f_{H}\left( x\right) =\langle H,x\rangle 
  	$ in this coordinate chart. A  similar description also exists for coordinate charts 
  	around the singularities $wH_{0}$, $w\in \mathcal{W}$.
  	As described in the lecture notes of San Martin \cite{SM}, the expressions for 
  	$f_{H}$ in these coordinate charts are polynomial functions,
  	in fact, homogeneous polynomials of degree $2$.\\

\begin{theorem}\cite{SM} \label{SanMartin}
The map $f_H$ is polynomial.
\end{theorem}	

\begin{proof}
  	If $X\in \mathfrak{n}_{H_{0}}^{+}$ then $\left[ X,H_{0}\right] =-%
  	\mathrm{ad}\left( H_{0}\right) X\in \mathfrak{n}_{H_{0}}^{+}$ and since $%
  	\mathfrak{n}_{H_{0}}^{+}$ is nilpotent, there exists  $n\geq 1$ such that
  	\begin{equation*}
  	e^{\mathrm{ad}\left( X\right) }H_{0}=H_{0}+g^{+}\left( X\right)
  	\end{equation*}%
  	where  $g^{+}:\mathfrak{n}_{H_{0}}^{+}\rightarrow \mathfrak{n}_{H_{0}}^{+}$ 
  	is the polynomial map
  	\begin{equation*}
  	g^{+}\left( X\right) =\sum_{k=1}^{n}\frac{1}{k!}\mathrm{ad}\left( X\right)
  	^{k}\left( H_{0}\right) .
  	\end{equation*}%
  	Now, if  $Y\in \mathfrak{n}_{H_{0}}^{-}$, then similarly $\left[ Y,H_{0}%
  	\right] =-\mathrm{ad}\left( H_{0}\right) Y\in \mathfrak{n}_{H_{0}}^{-}$. 
  	Moreover, $\mathfrak{n}_{H_{0}}^{-}$ is nilpotent and $\mathrm{ad}%
  	\left( Y\right) $ is  nilpotent. Thus, there exists $m\geq 1$ such that 
  	\begin{equation}
  	e^{\mathrm{ad}\left( Y\right) }e^{\mathrm{ad}\left( X\right)
  	}H_{0}=H_{0}+g^{-}\left( Y\right) +\sum_{j=0}^{m}\sum_{k=1}^{n}\frac{1}{j!}%
  	\frac{1}{k!}\mathrm{ad}\left( Y\right) ^{j}\mathrm{ad}\left( X\right)
  	^{k}\left( H_{0}\right)  \label{formontedesoma}
  	\end{equation}%
  	where $g^{-}\colon\mathfrak{n}_{H_{0}}^{-}\rightarrow \mathfrak{n}_{H_{0}}^{-}$ 
  	is the polynomial map
  	\begin{equation*}
  	g^{-}\left( Y\right) =\sum_{j=1}^{m}\frac{1}{j!}\mathrm{ad}\left( Y\right)
  	^{j}\left( H_{0}\right) .
  	\end{equation*}%
  	These expressions show that  
  	$f_{H}\left( e^{\mathrm{ad}\left(
  	Y\right) }e^{\mathrm{ad}\left( X\right) }H_{0}\right) $ is polynomial in $X$ and $Y$ given that $f_{H}$ is the restriction of the linear map 
  	$x\mapsto \langle H,x\rangle $. 
  	
  	Now, observe that 
  	$f_{H}\left(g^{-}\left( Y\right) \right) =0$ 
  	and consequently 
  	$$f_{H}\left( e^{\mathrm{ad}%
  		\left( Y\right) }e^{\mathrm{ad}\left( X\right) }H_{0}\right) =\langle
  	H,H_{0}\rangle +\langle H,\ast \rangle, $$ 
  	where $\ast$ is the part of the double summation in Equation  (\ref{formontedesoma}) belonging to $\mathfrak{h}$.
\end{proof}
  	
	\begin{example} 
  		
	Let 
	$\mathfrak{g}=\mathfrak{sl}\left(n+1,\mathbb{C}\right)$ 
	and choose $H_{0}$ such that the flag
	$\mathbb{F}_{H_{0}}$ 
	is the projective space $%
  	\mathbb P^{n}$. We take
	\begin{equation*}
	H_{0}=\mathrm{Diag}(n,-1,\ldots ,-1).
	\end{equation*}%
	In such a case, if $X\in \mathfrak{n}_{H_{0}}^{+}$, then
  	\begin{equation*}
  	X=\left( 
  	\begin{array}{cc}
  	0_{1\times 1} & v \\ 
  	0 & 0_{n \times  n }%
  	\end{array}%
  	\right),
  	\end{equation*}%
  	with  $v\in \mathbb{R}^{n}$, 
  	whereas $Y\in \mathfrak{n}_{H_{0}}^{-}$
  	if and only if  
  	$Y^{T}\in \mathfrak{n}_{H_{0}}^{+}$. If $X\in \mathfrak{n}%
  	_{H_{0}}^{+}$, 
  	then 
  	\begin{equation*}
  	e^{\mathrm{ad}\left( X\right) }H_{0}=H_{0}+\left[ X,H_{0}\right] =H_{0}-X.
	\end{equation*}%
	Consequently, if  $Y\in \mathfrak{n}_{H_{0}}^{-}$, then 
	\begin{eqnarray*}
  	e^{\mathrm{ad}\left( Y\right) }e^{\mathrm{ad}\left( X\right) }H_{0}
  	&=&H_{0}-X+Y-\left[ Y,X\right] +\frac{1}{2}\left[ Y,Y-\left[ Y,X\right] %
  	\right] \\
  	&=&-X+\left( H_{0}-\left[ Y,X\right] \right) +\left( Y-\frac{1}{2}\left[ Y,%
  	\left[ Y,X\right] \right] \right)
  	\end{eqnarray*}%,
  	because $\left[ Y,H_{0}\right] =Y$,  given that the last term has already been assembled
  	into the sum 
  	$\mathfrak{n}_{H_{0}}^{+}\oplus \mathfrak{z}%
  	_{H_{0}}\oplus \mathfrak{n}_{H_{0}}^{-}$ 
  	where $\mathfrak{z}_{H_{0}}$ is the centraliser of $H_{0}$.

	Now, if 
	$H=\Diag(\lambda_{1},\ldots ,\lambda_{n+1})$, 
	then $H$ is orthogonal to both
	$\mathfrak{n}_{H_{0}}^{+}$ 
	and $\mathfrak{n}%
	_{H_{0}}^{-}$ (Cartan--Killing); thus 
 	\begin{eqnarray*}
	f_{H}\left( e^{\mathrm{ad}\left( Y\right) }e^{\mathrm{ad}\left( X\right)
	}H_{0}\right) &=&\langle H,H_{0}\rangle -\langle H,\left[ Y,X\right] \rangle
	\\
	&=&\langle H,H_{0}\rangle +\langle H,\left[ X,Y\right] \rangle .
	\end{eqnarray*}%
	We write
	\begin{equation*}
	X=\left( 
	\begin{array}{cc}
	0_{1\times 1} & {\bf x} \\ 
	0 & 0_{n \times n }%
	\end{array}%
	\right) \qquad Y=\left( 
	\begin{array}{cc}
	0_{1\times 1} & 0 \\ 
	{\bf y}^{T} & 0_{ n \times  n}%
	\end{array}%
	\right)
	\end{equation*}%
	with ${\bf x}=\left( x_{1},\ldots ,x_{n}\right) $ and 
	${\bf y}=\left( y_{1},\ldots ,y_{n}\right)$. 
	Since the Cartan--Killing product is $\langle A,B\rangle =2(n+1)\mathrm{tr}%
	\left(AB\right)$
	for matrices in $\mathfrak{sl}(n+1)$, it follows that 
	\begin{align*}
	f_{H}\left( e^{\mathrm{ad}\left( Y\right) }e^{\mathrm{ad}\left( X\right)
	}H_{0}\right) =&\\
	=2(n+1)\left[\mathrm{tr}(H H_{0})+\right.&\left. \left(
	\lambda_{1}-\lambda_{2}\right) x_{1}y_{1}+\cdots +\left( \lambda_{1}-\lambda_{n+1}\right)
	x_{n}y_{n}\right]
	\end{align*}
	which is a degree $2$ polynomial.

	Similar calculations go through if we replace  $H_{0}$ by $w\cdot H_{0}$: if $%
	w=\left( 1j\right) $, then $\mathfrak{n}_{w\cdot H_{0}}^{+}$ consists of matrices having nonzero entries only on row $j$ ($0$
	in the diagonal), whereas 
	$\mathfrak{n}_{w\cdot H_{0}}^{-}$ consists of nonzero 
	entries only in column $j$. 
	Thus, if  $\left( x_{1},\ldots ,x_{n}\right) $ and $\left(
y_{1},\ldots ,y_{n}\right) $ are the entries on row and column $j$,
	respectively, then
	\begin{align*}
	f_{H}\left( e^{\mathrm{ad}\left( Y\right) }e^{\mathrm{ad}\left( X\right)
	}\left( w \cdot H_{0}\right) \right) =&\\
	=2(n+1)\left[\mathrm{tr}\left( H (w \cdot H_{0})\right)\right.
	+ &\left.\left(  \lambda_{j}-\lambda_{1}\right) x_{1}y_{1}+\cdots +\left(
	\lambda_{j}-\lambda_{n+1}\right) x_{n}y_{n}\right] .
	\end{align*}
	\end{example}

\begin{example}
Let 
$H=\Diag(1,0,-1), H_0=\Diag(2,-1,-1) \in \mathfrak{sl}(3, \mathbb{C})$ and choose 
$w = (123) \in \mathcal{W}$.
Then $w \cdot H_0=\Diag(-1,2,-1)$ and $w^2 \cdot H_0=\Diag(-1,-1,2)$, so that 
\begin{align*}
\langle H,H_0\rangle & =2(n+1)\mathrm{tr}%
	\left(H H_0\right)= 18, \\
	\langle H,w \cdot H_0\rangle & = 2(n+1)\mathrm{tr}%
	\left(H  (w \cdot H_0)\right)= 0, \\
\langle H,w^2 \cdot H_0\rangle & =2(n+1)\mathrm{tr}%
	\left(H  (w^2 \cdot H_0)\right)=  -18. 
\end{align*}
	Note that the constant term of the expression changes according to the choice of critical points.
\end{example}

\begin{definition}
We call the expression of $f_H$ written on a chart around $H_0$ the \textbf{Lie potential} on
 $\mathcal{O}_n \subset \mathfrak{sl}(n+1,\mathbb{C})$. 
It is given by 
\begin{equation} \label{liepot}
\boxed{
\textbf{f}_H (H_0)= 
\mathrm{tr}(H H_{0}) + (
	\lambda_{1}-\lambda_{2}) x_{1}y_{1}+\cdots +( \lambda_{1}-\lambda_{n+1})
	x_{n}y_{n}.
}\end{equation}
\end{definition}	

  	\vspace{12pt}%
  
  	\noindent%
	\begin{example} 
	Consider 
	$\mathfrak{g}= \mathfrak{sl}\left( n+1,\mathbb{C}\right)$ 
	and choose $H_{0}$ so that the flag $\mathbb{F}_{H_{0}}$ is the  Grassmannian 
	$\mathrm{Gr}_{k}\left( n+1\right)$. 
	Calculations are similar to the ones for the projective space. 
	The main difference is that the decomposition is given by blocks of size $k$ and $n-k+1$, so that the elements of 
	$\mathfrak{n}_{H_{0}}^{+}$ are matrices of the form
	\begin{equation*}
	\left( 
	\begin{array}{cc}
	0_{k\times k} & v \\ 
	0 & 0_{\left( n-k+1\right) \times \left( n-k+1\right) }%
	\end{array}%
	\right) .
	\end{equation*}%
	Similarly, 
	\begin{eqnarray*}
	e^{\mathrm{ad}\left( Y\right) }e^{\mathrm{ad}\left( X\right) }H_{0}
	&=&H_{0}-X+Y-\left[ Y,X\right] +\frac{1}{2}\left[ Y,Y-\left[ Y,X\right] %
	\right]\\   
	&=&-X+\left( H_{0}-\left[ Y,X\right] \right) +\left( Y-\frac{1}{2}\left[ Y,%
	\left[ Y,X\right] \right] \right)  \notag
	\end{eqnarray*}%
	and $f_{H}\left( e^{\mathrm{ad}\left( Y\right) }e^{\mathrm{ad}\left( X\right)}H_{0}\right) = \langle H,H_{0}-\left[ Y,X\right] \rangle$ 
	is a degree $2$ polynomial.
  	\end{example}

  	These 2 examples can be easily generalised to all cases when the eigenvalues of 
  	$\mathrm{ad}\left( H_{0}\right)$ are $0$ and $\pm 1$. 
  	Analogous calculations show that 
  	\begin{equation*}
  	e^{\mathrm{ad}\left( Y\right) }e^{\mathrm{ad}\left( X\right)
  	}H_{0}=-X+\left( H_{0}-\left[ Y,X\right] \right) +\left( Y-\frac{1}{2}\left[
  	Y,\left[ Y,X\right] \right] \right)
  	\end{equation*}%
  	and hence $f_{H}\left( e^{\mathrm{ad}\left( Y\right) }e^{\mathrm{ad} 	\left( X\right) }H_{0}\right) =\langle H,H_{0}-\left[ Y,X\right] \rangle$
  	is a degree $2$ polynomial in the variables $X$ and $Y$.

\begin{remark}
 	In further generality, given a complex Lie algebra,
	we may choose  a set of positive roots $\Pi ^{+}$ and simple roots $\Sigma \subset \Pi ^{+}$
with corresponding Weyl chamber is $\mathfrak{a}^{+}$.
	A subset $\Theta \subset \Sigma $ defines a parabolic subalgebra $%
\mathfrak{p}_{\Theta }$ with parabolic subgroup $P_{\Theta }$ and a flag manifold
$\mathbb{F}_{\Theta }=G/P_{\Theta }$.
	An element
$H_{\Theta }\in \mathrm{cl}\mathfrak{a}^{+}$ is characteristic
for $\Theta \subset \Sigma $ if $\Theta =\{\alpha \in \Sigma :\alpha \left(
H_{\Theta }\right) =0\}$.
Let
$Z_{\Theta }=\{g\in G:\mathrm{Ad}\left( g\right) H_{\Theta }=H_{\Theta
}\}$ be the centraliser in $G$ of the characteristic element $H_{\Theta }$.
The adjoint orbit 
$\mathcal{O}\left( H_{\Theta }\right)
=\mathrm{Ad}\left( G\right) \cdot H_{\Theta }\approx G/Z_{\Theta }$ 
of the characteristic element $H_{\Theta }$ is diffeomorphic to the cotangent bundle 
$T^{\ast }\mathbb{F}_{\Theta }$
\cite[Thm.\thinspace 2.1]{GGS2}.
\end{remark}

\begin{remark}\label{comp}
In the specific case of the minimal orbit, there is a very useful compactification to 
a product or projective spaces, which goes as follows.
The diagonal action of  $\mathrm{SL}\left( n+1,\mathbb C \right) $ in
 $\mathbb{P}^{n}\times \mathrm{Gr}\left(n, n+1\right)\simeq	\mathbb{P}^{n}\times \mathbb P^n $ has two orbits, 
	an open one and a closed one. The open orbit is isomorphic to the adjoint orbit
	$\mathcal O_n=\Ad\left( G\right) \cdot H_{\Theta }$ with $H_{\Theta }=%
	\Diag(n,-1,\ldots ,-1)$ and is formed by the pairs of transversal elements in
	$\mathbb{P}^{n}\times \mathrm{Gr}\left(n, n+1\right) $.
The closed orbit is isomorphic to the flag $F_{\Theta \cap
		\Theta ^{\ast }}$. Since $\Theta \cap \Theta ^{\ast }$ is the complement of 
	$\{\alpha _{12},\alpha _{n,n+1}\}$ it follows that $F_{\Theta \cap
		\Theta ^{\ast }}=F\left( 1,n\right) $.
\end{remark}

\section{Toric potential vs Lie potential}
 Cotangent bundles of projective spaces and their Lagrangian skeleta were studied in \cite{BGRS1} 
 from the viewpoint of their duality with bundles on  collars of local surfaces. 
Starting with a Hamiltonian action of $\mathbb T = \mathbb C^*$   on $T^\ast\mathbb{P}^n$ 
expressed in the open chart $V_0=\{x_0\neq 0\}$ by
$$\mathbb T \cdot V_0=
\left\{[1,t^{-1} x_1,\dots,t^{-n}x_n],(ty_1,\dots, t^ny_n)\right\}$$
 we obtain a Hamiltonian vector field $T^\ast\mathbb{P}^n$
\[
	X(x_1, \ldots, x_n, y_1, \ldots, y_n)
	= (-x_1,\dots,-nx_n,y_1,\dots, ny_n),
\]
as described in \cite[Prop.\thinspace 3.1]{BGRS1}.
 Taking $T^\ast\mathbb{P}^n$ with the canonical symplectic form $\omega$, 
the corresponding potential $h$ therefore satisfies 
\[
dh(Z)=\omega(X,Z),
\]
for all vector fields $Z \in \mathfrak{X}(T^\ast\mathbb{P}^n)$.
In $V_0$ coordinates, 
writing 
\[
{\bf x}= (x_1, \dots, x_n), 
{\bf y}= (y_1, \dots, y_n), 
{\bf a}= (a_1, \dots, a_n), 
{\bf b}= (b_1, \dots, b_n),
\]
we obtain a differential equation
\begin{align*}
\left(
\begin{matrix}
	\displaystyle\frac{\partial h}{\partial {\bf x}}   \displaystyle\frac{\partial h}{\partial {\bf y}} 
\end{matrix}
\right)
\left(
\begin{matrix}
	{\bf a}  \\ {\bf b} 
\end{matrix}
\right)
 & =
\sum_{i=1}^n dx_i \wedge dy_i \left( (-x_1, \ldots , -nx_n , y_1, \ldots , ny_n),
(  a_1,\ldots ,a_n , b_1,\ldots, b_n )\right) \\
 & = -2\sum ix_ib_i+iy_ia_i,
\end{align*}
whose solutions are:
\begin{equation}\label{toricpoteq}
{\bf h}_c = c -2x_1y_1 - \ldots -2nx_ny_n = c + \sum_{i=1}^n -2ix_iy_i,
\end{equation}
for $c \in \mathbb{C}$.
See \cite{BGRS1} for the expression of $h_c$ on the other charts. 
In Section \ref{Sensational} we will consider $n=1$, in which case there are two charts.

\begin{definition}\label{toricpot}
We call ${\bf h}_c$ a \textbf{toric potential} on $T^*\mathbb{P}^n$.
\end{definition}

\begin{theorem}\label{coinc} 
For each $n \in \mathbb{N}$, there exist matrices $H, H_0 \in \mathfrak{sl}\left(n+1,\mathbb{C}\right)$ and a constant $c \in \mathbb{C}$ such that the Lie potential on the minimal adjoint orbit $\mathcal{O}_n$ and the toric potential on the cotangent bundle $T^*\mathbb P^{n}$ coincide on dense open 
charts, that is  ${\bf f}_H = {\bf h}_c$.
\end{theorem}

\begin{proof}
Take 
$H_{0}=\Diag(n,-1,\ldots ,-1)\in \mathfrak{sl}\left(n+1,\mathbb{C}\right)$ ,

	\[
	H=\begin{cases}
 \Diag\left(-n, \ldots, -4, -2, 0, 2, 4, \ldots, n \right) 
\quad \textnormal{if} \,\, n \textnormal{ is even,}\\
 \Diag\left( -n, \ldots, -3, -1, 1, 3, \ldots, n \right)
\quad \textnormal{if} \,\, n \textnormal{ is odd}
 \end{cases}
	\] 
and $c=-n^2-n$.
	
On the Lie side, we have:
\[
\textbf{f}_H (H_0)= 
\mathrm{tr}(H H_{0}) + (
	\lambda_{1}-\lambda_{2}) x_{1}y_{1}+\cdots +( \lambda_{1}-\lambda_{n+1})
	x_{n}y_{n},
\]
where the eigenvalues of $H=\Diag(\lambda_1, \ldots, \lambda_{n+1})$ 
satisfy 
$\lambda_1 - \lambda_j = -2(j-1)$, 
so that
\[
\textbf{f}_H (H_0)= 
-n^2-n 
-2 x_{1}y_{1}
- \cdots 
- 2n x_{n}y_{n}.
\]

On the toric side,
\begin{align*}
{\bf h}_c & = c -2x_1y_1 - \ldots -2nx_ny_n \\
	& = -n^2-n 
	- 2x_1y_1 - \cdots - 2nx_ny_n.
\end{align*}
\end{proof}

By Theorem \ref{coinc} the value $c=-n^2-n$ is preferable for our purposes here.
 Therefore, using
coordinates $([1, x_1],y_1)$  for the cotangent bundle  $T^*\mathbb P^1$,
we write the potential 
$${\bf h}_{-2}\colon T^*\mathbb P^1\rightarrow \mathbb C$$
as
$${\bf h}_{-2}=- 2x_1y_1-2.$$

\begin{definition}
Two Landau--Ginzburg models are called {\bf birationally equivalent} if their domains are birationally equivalent varieties and their potentials coincide on Zariski open sets.
\end{definition}

A rigorous version of our theorem may be stated as:

\begin{theorem}\label{bit}
For each $n \in \mathbb{N}$, there exist matrices $H, H_0 \in \mathfrak{sl}\left(n+1,\mathbb{C}\right)$ and a constant $c \in \mathbb{C}$ such 
that the LG models $(\mathcal O_n, f_H)$ and $(T^*\mathbb P^{n}, h_c)$ are birationally equivalent.
\end{theorem}

\begin{proof}
By Theorem \ref{coinc}, the expressions of ${\bf f}_H$ and ${\bf h}_c$ for the potentials $f_H$ and $h_c$ coincide on Zariski open sets. 
\end{proof}

\section{Deformations and their mirrors} \label{Sensational}

\begin{example}[$\LG_0 \longrightsquigarrow \LG_1$] \label{0to1}
We consider $T^*\mathbb P^1$ with coordinates $[1,x],y$ and the family of potentials 
$$\LG_t = t\left(x+\frac{1}{x}+ \frac{y^2}{x}\right)+(1-t) (2x).$$
Then we have that the initial LG model 
$LG_0$ is selfdual as shown in Example \ref{selfdual}, while the final LG model 
is defined by $\LG_1 = (T^*\mathbb P^1,2x)$ which has the toric data
	$$\Div= \left(\!\!\!
	\begin{array}{rr} 
	1& 0 \\
	-1 &2\\
	0&1
	\end{array} \!\right), 
	\quad 
	\Mon = (1 \, 0).$$
Now, inverting the matrices by toric duality gives us the $\LG_1^\vee$ model 
	$$\Div=(1 \, 0) 
	\quad 
	\Mon = \left( \!\!\!
	\begin{array}{rr} 
	1& 0 \\ 
	-1 &2\\
	0&1
	\end{array}\!\!\right)$$
defined on a variety with torus $(\mathbb C^*)^2$ and a single divisor, in this case a $\mathbb C^* \times \mathbb C$.
We conclude that this family takes a selfdual LG model to another very far from selfdual.

$$\begin{array}{clc}\LG_0& \longrightsquigarrow &\LG_1\\
\rotatebox[origin=c]{90}{=} & &\not\sim\\
\LG_0^\vee& \longrightsquigarrow &\LG_1^\vee .
\end{array}$$

\end{example}

Our next objective is to describe how duality works for the deformation family 
$$\LG_1= (T^*\mathbb P^1,h_c) \longrightsquigarrow \LG_2= (\mathcal O_2, f_H).$$
We will use the deformation of the Hirzebruch surfaces $\mathbb F_2$ to $\mathbb F_0$ in Lemma \ref{deform}, 
extending it to a deformation of partially compactified Landau--Ginzburg models in Example \ref{partial}
$$(\mathbb F_2,\bf h) \longrightsquigarrow (\mathbb F_0,\mathbf f).$$

The Hirzebruch surface $\mathbb F_2$ embeds in  $\mathbb P^1 \times \mathbb P^3$,
 as  the set of points 
  $$([x_0,x_1], [y_0, y_1, y_2, y_3])  \in \mathbb P^1 \times \mathbb P^3 $$ 
satisfying the equation
\[
x_0 (y_1, y_2) = x_1 (y_2, y_3), 
\]
see \cite{Ma}. 
We can cover $\mathbb{F}_2$ with four coordinate charts, 
$U = \{ (z, u) \in \mathbb C^2 \}$, 
$V = \{ (\xi, v) \in \mathbb C^2 \}$, 
$U' = \{ (z, \mu) \in \mathbb C^2 \}$,
$V = \{ (z, \eta) \in \mathbb C^2 \}$
with the following changes of coordinates: 
\[
\xi = z^{-1}, \quad v = z^2u, \quad \mu = u^{-1}, \quad \eta = v^{-1}.
\]

 We will make use of the following result:
\begin{lemma}\label{deform}
$\mathbb{F}_2$ deforms to $\mathbb{F}_0$.
\end{lemma}

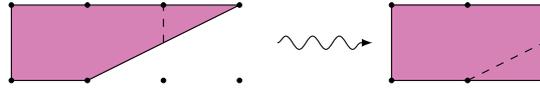
\begin{figure}[h]
\centering
\label{f2tof0}

\begin{tikzpicture}
\begin{scope}[shift={(0,0)}] 
\filldraw[Thistle] (0,1) -- ++(0,-1) -- ++(1,0) -- ++(2,1);

\foreach \x in {0,...,3}
		{	\foreach \y in {0,...,1}
			{
			\filldraw (\x,\y) circle (0.03cm);
			}
		}

\draw (0,1) -- ++(0,-1) -- ++(1,0) -- ++(2,1) -- (0,1);
\draw[dashed] (2,1/2) -- ++(0,1/2);
\end{scope}

\begin{scope}[shift={(5,0)}] 
	\filldraw[Thistle] (0,0) -- ++(2,0) -- ++(0,1) -- ++(-2,0) -- ++(0,-1);
	\foreach \x in {0,...,2}
		{
		\foreach \y in {0,...,1}
			{
			\filldraw (\x,\y) circle (0.03cm);
			}
		}
	\draw (0,0) -- ++(2,0) -- ++(0,1) -- ++(-2,0) -- ++(0,-1);
	\draw[dashed] (1,0) -- ++(1,1/2);
\end{scope}

\draw[->, -latex, decorate, decoration={snake, post length=1pt}]
(3.5,0.5) -- ++(1.25,0); 
\end{tikzpicture}

\caption{$\mathbb{F}_2$ deforms to $\mathbb{F}_0$}
\end{figure}

\begin{proof}
The deformation family $\mathcal{M}$ taking $\mathbb{F}_2$ to $\mathbb{F}_0$ is given by 
$$\{x_0 (y_1, y_2) = x_1 (y_2 + t y_0, y_3)\} = M_t \subset \mathbb{P}^1 \times \mathbb{P}^3. $$
In coordinates:
\begin{align*}
(z, u, t)		&\mapsto ([1, z], [1, z^2 u + t z, zu, u], t) \\
(z, \mu, t)		&\mapsto ([1, z], [\mu, z^2 + tz\mu, z, 1], t) \\
(\xi, v, t)		&\mapsto ([\xi, 1], [1, v, \xi v - t, \xi^2 v - t \xi], t). \\
(\xi, \eta, t)	&\mapsto ([\xi, 1], [\eta, 1, \xi - t\eta, \xi^2 - t \xi\eta], t). 
\end{align*}

The deformation from $\mathbb{F}_2$ to $\mathbb{F}_0$ is given by 
$$\{x_0 (y_1, y_2) = x_1 (y_2 + t y_0, y_3)\} = M_t \subset \mathbb{P}^1 \times \mathbb{P}^3, $$
\begin{align*}
x_0y_1 & = x_1y_2 + tx_1y_0 \\
x_0y_2 & = x_1y_3
\end{align*}

\end{proof}

\begin{lemma} \label{deform.fam}
$T^*\mathbb P^1$ deforms to $\mathcal O_2$.
\end{lemma}

\begin{figure}[h]
\centering
\begin{tikzpicture}
%mid left
\begin{scope}[shift={(0,0)}] 
\filldraw[Thistle] (0,1) -- ++(0,-1) -- ++(1,0) -- ++(2,1);

\foreach \x in {0,...,3}
		{	\foreach \y in {0,...,1}
			{
			\filldraw (\x,\y) circle (0.03cm);
			}
		}

\draw (0,1) -- ++(0,-1) -- ++(1,0) -- ++(2,1);

\foreach \i in {0,...,3}
	{
	\filldraw[white] (\i,1) circle (0.02cm);
	}	
\end{scope}

\begin{scope}[shift={(5,0)}] 
	\filldraw[Thistle] (0,0) -- ++(2,0) -- ++(0,1) -- ++(-2,0) -- ++(0,-1);
	\foreach \x in {0,...,2}
		{
		\foreach \y in {0,...,1}
			{
			\filldraw (\x,\y) circle (0.03cm);
			}
		}
	\draw (0,1) -- ++(0,-1) -- ++(1,0);
	\draw (2,0) -- ++(0,1);
\foreach \i in {0,...,3}
	{
	\filldraw[white] (\i,1) circle (0.02cm);
	}
	\filldraw[white] (2,0) circle (0.02cm);

\end{scope}
\begin{scope}[shift={(9,0)}] 
	\filldraw[Thistle] (0,0) -- ++(2,0) -- ++(0,1) -- ++(-2,0) -- ++(0,-1);
	\foreach \x in {0,...,2}
		{
		\foreach \y in {0,...,1}
			{
			\filldraw (\x,\y) circle (0.03cm);
			}
		}
	\draw (0,0) -- ++(2,0) -- ++(0,1) -- ++(-2,0) -- ++(0,-1);
\draw[dashed, white] (0,0) -- (2,1);
	\filldraw[white] (0,0) circle (0.02cm);
	\filldraw[white] (2,1) circle (0.02cm);
\end{scope}

\draw[->, -latex, decorate, decoration={snake, post length=1pt}] 
(3.5,0.5) -- ++(1.25,0); 

\draw[->, -latex, decorate, decoration={snake, post length=1pt}] 
(7.5,0.5) -- ++(1.25,0); 

\end{tikzpicture}
\label{tp1to02}
\caption{$T^*\mathbb{P}^1$ deforms to $\mathcal{O}_2$}
\end{figure}
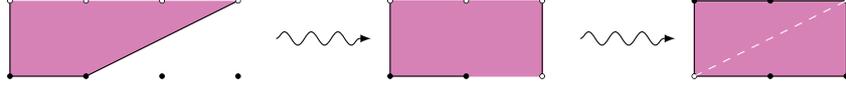

\begin{proof}
We will use the family 
 $\pi \colon \mathcal{M} \to \mathbb C$ given in Lemma \ref{deform} deforming the Hirzebruch surface $\mathbb F_2$ to $\mathbb F_0$,
  written as the points $([x_0,x_1], [y_0, y_1, y_2, y_3], t) \in \mathbb P^1 \times \mathbb P^3 \times \mathbb C$ satisfying the equation
\[
x_0 (y_1, y_2) = x_1 (y_2 + t y_0, y_3).
\]
	Here $\pi^{-1} (0) \simeq \mathbb F_2$ and for $t \neq 0$ we have that $\pi^{-1} (t) \simeq \mathbb F_0 \simeq \mathbb P^1 \times \mathbb P^1$.

The orbit $\mathcal O_2 $ may be written in coordinate charts as $U = \{ (z, u) \in \mathbb C^2 \}$ and $V = \{ (\xi, v) \in \mathbb C^2 \}$ with transition function
\[
(\xi, v) \mapsto (z^{-1}, z^2 u + t z),
\]
where we have $z= \frac{x_1}{x_0}$, $\xi = \frac{x_0}{x_1}$, $u=y_3$, and $v=y_1$.
The family $M \to \mathbb C$ deforming $T^*\mathbb P^1$ to $\mathcal O_2 $ admits an embedding 
\begin{equation}\label{uno}{j}\colon M \to \mathcal M\end{equation} given by
\begin{align*}
(z, u, t) &\mapsto ([1, z], [1, z^2 u + t z, z u, u], t) \\
(\xi, v, t) &\mapsto ([\xi, 1], [1, v, \xi v - t, \xi^2 v - t \xi], t).
\end{align*}
Note that in these coordinates,   $t=0$ and $y_1=y_2=y_3=0$ give the zero section of $T^*\mathbb P^1$.\\

\begin{claim}\label{infsection} The section at infinity $\sigma_\infty$ is:
\begin{align*}
\sigma_\infty(z,t) &= ([1, z], [0,z^2,z,1], t)\quad \text{on}\quad  U, \\
\sigma_\infty(\xi,t) &= ([\xi, 1], [0,1,\xi,\xi^2], t)\quad \text{on} \quad V, 
\end{align*}
\end{claim}

\noindent{\it Proof of the claim}. Note that
\begin{align*}
[1, z^2 u + t z, z u, u] & = [1/u, z^2 + tz/u, z, 1] \\
[1, v, \xi v - t, \xi^2 v - t \xi] & = [1/v, 1, \xi -t/v, \xi^2 -t\xi/v] 
\end{align*}
So, by taking the limit when $u \to \infty$ and $v \to \infty$, i.e., $1/u \to 0$ and $1/v \to 0$ we obtain the section at infinity:
\begin{align*}
\lim_{u \rightarrow \infty}(z, u, t) & \mapsto ([1, z], [0,z^2,z,1], t), \\
\lim_{v \rightarrow \infty}(\xi,v, t) &\mapsto ([\xi, 1], [0,1,\xi,\xi^2], t).
\end{align*}

\end{proof}

\begin{example} \label{partial}($\LG_1\longrightsquigarrow \LG_2$)
We now revisit the family 
   $\LG_1\longrightsquigarrow \LG_2$ presented in Example \ref{ex2}, commenting on the dual family.

\begin{figure}[h]
\centering
\begin{tikzpicture}

\filldraw[Thistle] (0,1) -- ++(0,-1) -- ++(1,0) -- ++(2,1);

\foreach \x in {0,...,3}
		{	\foreach \y in {0,...,1}
			{
			\filldraw (\x,\y) circle (0.03cm);
			}
		}

\draw (0,1) -- ++(0,-1) -- ++(1,0) -- ++(2,1);

\foreach \i in {0,...,3}
	{
	\filldraw[white] (\i,1) circle (0.02cm);
	}	

\begin{scope}[shift={(5,0)}] 
	\filldraw[Thistle] (0,0) -- ++(2,0) -- ++(0,1) -- ++(-2,0) -- ++(0,-1);
	\foreach \x in {0,...,2}
		{
		\foreach \y in {0,...,1}
			{
			\filldraw (\x,\y) circle (0.03cm);
			}
		}
	\draw (0,0) -- ++(2,0) -- ++(0,1) -- ++(-2,0) -- ++(0,-1);
\draw[dashed, white] (0,0) -- (2,1);
	\filldraw[white] (0,0) circle (0.02cm);
	\filldraw[white] (2,1) circle (0.02cm);
\end{scope}

\draw[->, -latex, decorate, decoration={snake, post length=1pt}]
(3.5,0.5) -- ++(1.25,0);

\draw[->, -latex, decorate, decoration={snake, post length=1pt}]
(2.5,-1) -- ++(1.25,0); 

\coordinate (A) at (0.5,-1);
\coordinate (B) at (5.5,-1);

\node at (A) {$2x$};
\node at (B) {$2x$};

\end{tikzpicture}
\caption{\sc $\LG_1$ to $ \LG_2$}
\end{figure}

Using Definition \ref{toricpot} with $c=-2$, $n=1$ and coordinates
 $[1, x_1],y_1$
 for the cotangent bundle 
 $T^*\mathbb P^1$,
the potential 
${\bf h}_{-2}\colon T^*\mathbb P^1\rightarrow \mathbb{C}$
can be written as
$$
{\bf h}_{-2}([1, x_1],y_1)=- 2x_1y_1-2, 
\quad {\bf h}_{-2}([0, 1],v_1) = -2.
$$
As seen in Example \ref{0to1}, toric duality gives the dual $\LG_1^\vee$ 
as defined over a variety having a single divisor with a potential formed
by 3 monomials, hence not isomorphic to $\LG_1$.\\

We now discuss  the dual of the Landau--Ginzburg model 
$\LG_2 = (\mathcal O_1, f_H),$ 
where 
$\mathcal O_1$ is the semisimple adjoint orbit of $\mathfrak{sl}(2,\mathbb C)$, 
together with the potential $f_H$ given by the height function $f_H(x,y,z) =2x$. 
To start with, toric duality can not be applied in this case, because the orbit $\mathcal O_1$ 
is not a toric variety. So, an entire new recipe is needed. 

The recipe for intrinsic mirrors by Gross and Siebert \cite{GS} produces the desired dual Landau--Ginzburg model, which was described in \cite{G}. We summarise the construction. 
 
The compactification mentioned in Remark \ref{comp}, applied to 
 the case when $n=1$ gives  $\mathbb P^1\times \mathbb P^1 = \mathcal O_1 \cup \Delta$
where $\Delta $ is the diagonal.	
To construct the intrinsic mirror, we start with the pair $(\mathbb P^1\times \mathbb P^1,\Delta)$.
Since Gross and Siebert work under the hypothesis of log geometry, 
we need a pair $(X,D)$, where 
$D$ is an anticanonical divisor.
In this case, 
 an anticanonical divisor has type $(2,2)$. 
So, we choose  an additional divisor 
$\Delta'$ of type $(1,1)$ such that $\Delta \cap \Delta'= \{p_1,p_2\}$ consists of 2 points. 

Let $\widetilde{X}$ denote the blow up of $\mathbb P^1\times \mathbb P^1$
at these 2 points.
We then study the pair 
$$(\widetilde{X},D) \qquad \text{with}  \qquad D= D_1+D_2+ D_3 + D_4$$
where $D_2$ and $D_4$ are the proper transforms of 
$\Delta'$  and $\Delta$ respectively, and $D_1$ and $D_3$ are the exceptional curves obtained from blowing
up the 2 points; thus
$D_1^2=D_3^2=-1.$		

Calculating theta functions following the Gross--Siebert intrinsic mirror symmetry recipe, 
\cite{G}  arrives at the variety given by
\begin{equation}\label{mir}
\vartheta_2\vartheta_4=\alpha\vartheta_1+\beta\vartheta_1^{-1}+\gamma
\end{equation}
with potential $\vartheta_2$.

Changing to more standard algebraic coordinates
 $ x \ce \vartheta_1, y \ce \vartheta_2$ and compactifying the direction $\vartheta_4$ to $\mathbb P^1$ 
 with coordinates $[u:v]$,
our problem is then to study the surface 
$S \subset \mathbb C^*\times \mathbb C\times \mathbb P^1 = \{(x,y),[u:v]\}$  given by
$$S\ce \{uy=v(x+1+1/x)\}$$
with potential $$y=v(x+1+1/x)/u.$$

The Landau--Ginzburg model $\LG_2^\vee = (S, y) $ is the dual of $\LG_2$.
Denoting by $D_{sg}\LG_2^\vee$ the Orlov category of singularities of $\LG_2^\vee$, \cite[Sec.\thinspace 3.3]{G}  proves the 
equivalence 
$$D_{sg}\LG_2^\vee \equiv \Fuk \LG_2.$$

\end{example}

Note that here, $\LG_2^\vee$ is quite different from $\LG_2$, for instance, the potential $y$ has the 2 singularities 
of $\LG_2^\vee$ occurring on the same fibre of the potential, unlike the potential $2x$ of $\LG_2$ 
for which the singularities occur in different fibres, as described in Figure \ref{sphere}.\\

So that the mirror of the Landau--Ginzburg model geometrically resembles the rotation by 90 degrees of the original one.\\

\begin{figure}\label{sphere}
\centering
\begin{tikzpicture}
\shade[top color=Aquamarine, bottom color=Salmon] (4,2) circle (2cm);
\draw (4,2) ellipse (2cm and .5cm);
\draw (4,2) circle (2cm);
\draw[thick,->] (4,-.5) -- (4,-1.5) node[anchor=south west]{$y$};
\draw[thick,->] (6.5,2) -- (7.5,2)node[anchor=south east ]{$2x$};
\node at (4.1,4.23) {N};
\node at (4.1, -.2) {S};
\draw[fill] (4,4) circle (.05cm);
\draw[fill] (4,0) circle (.05cm);
\draw[fill] (8,4) circle (.05cm);
\draw[fill] (8,0) circle (.05cm);
\draw[fill] (4,-2) circle (.05cm);
\draw[fill] (8,0) -- (8,4);
\draw[fill] (2,-2) -- (6,-2);
\end{tikzpicture}
\caption{$\LG_2$ and its mirror}
\end{figure}
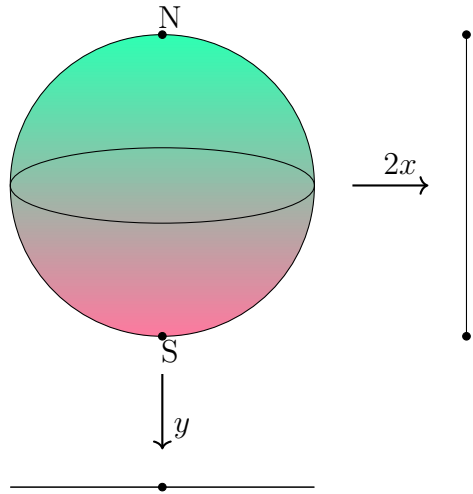

\noindent {\bf Acknowledgments:}
I thank Elizabeth Gasparim, and upon request, give her the following absurdly overstated acknowledgement:  
without her help this work would have provoked
\lq\lq a catastrophe, a tragedy, a cataclysmic, apocalyptic, monumental calamity\rq\rq, (Sir Humphrey Appleby, \cite[18:30]{YM}).

 The author thanks Luiz A. B. San Martin for providing many useful suggestions and for the proof of Theorem \ref{SanMartin}. 
 
 The author was partially supported by MathAmSud ANID project MATH2020011.
 The author was supported by Grant 2021/11750-7 S\~ao Paulo Research Foundation - FAPESP.

\end{document}